\documentclass[proceedings,submission]{dmtcs}

\usepackage{subfigure}

\usepackage{amsmath}
\usepackage{amssymb}
\usepackage{graphicx,xspace}
\usepackage{enumitem}

\usepackage[T1]{fontenc}
\usepackage[utf8]{inputenc}

\usepackage{hyperref}

\newtheorem{lemma}{Lemma}
\newtheorem{theorem}[lemma]{Theorem}
\newtheorem{corollary}[lemma]{Corollary}
\newtheorem{proposition}[lemma]{Proposition}
\newtheorem{problem}[lemma]{Problem}

\newtheorem{fact}[lemma]{Fact}

\newcommand{\R}{\mathbb{R}}

\newcommand{\N}{\mathbb{N}}
\newcommand{\Y}{\mathbb{Y}}

\newcommand{\F}{\mathcal{F}}
\newcommand{\p}{\mathbf{p}}
\newcommand{\q}{\mathbf{q}}
\newcommand{\Sym}{\mathfrak{S}}

\newcommand{\M}{\mathcal{M}}

\newcommand{\ee}{e$^2$}

\newcommand{\eeee}{e$^3$}

\DeclareMathOperator{\Tr}{Tr}

\DeclareMathOperator{\contents}{contents}

\RCSdef$Revision: 1.3 $\endRCSdef
\rcsMajMin

\author{Maciej Do\l\k{e}ga \addressmark{1}\and
Valentin F\'eray\addressmark{2}\and
Piotr \'Sniady\addressmark{1}\thanks{Supported by the
MNiSW research grant P03A 013 30, by the
EU Research Training Network ``QP-Applications'', contract HPRN-CT-2002-00279
and by the EC Marie Curie Host Fellowship for the Transfer of Knowledge
``Harmonic Analysis, Nonlinear Analysis and Probability'', contract
MTKD-CT-2004-013389. P\'S thanks Marek Bo\.zejko, Philippe Biane, Akihito Hora,
Jonathan Novak, \'Swiatosław Gal and Jan Dymara for several stimulating
discussions during various stages of this research project.
}}

\address{\addressmark{1} Institute of Mathematics,
University of Wroclaw,  \mbox{pl.\ Grunwaldzki~2/4,} 50-384
Wroclaw, Poland \\
\addressmark{2}The Gaspard--Monge Institute of Electronics and Computer Science,
University of Marne-La-Valle\'e Paris-Est, 77454 Marne-la-Vall\'ee Cedex 2,
France}

\title{Characters of symmetric
groups in terms of free cumulants and Frobenius coordinates}


\begin{document}
\maketitle
\begin{abstract}
Free cumulants are nice and useful functionals of the shape of a Young diagram,
in particular they give the asymptotics of normalized characters of symmetric
groups $\mathfrak{S}(n)$ in the limit $n\to\infty$. We give an explicit
combinatorial formula for normalized characters of the symmetric groups in terms
of free cumulants. We also express characters in terms of Frobenius coordinates.
Our formulas involve counting certain factorizations of a given permutation. The
main tool are Stanley polynomials which give values of characters on
multirectangular Young diagrams.

\noindent {\bf R\'esum\'e.}
Les cumulants libres sont des fonctions agr\'eables et utiles sur l'ensemble des
diagrammes de Young, en particulier, ils donnent le comportement asymptotiques
des caract\`eres normalis\'es du groupe sym\'etrique $\mathfrak{S}(n)$ dans la
limite $ n \to \infty$. Nous donnons une formule combinatoire explicite pour les
caract\`eres normalis\'es du groupe sym\'etrique en fonction des cumulants
libres. Nous exprimons \'egalement les caract\`eres en fonction des
coordonn\'ees de Frobenius. Nos formules font intervenir le nombre de certaines
factorisations d'une permutation donn\'ee. L'outil principal est la famille de
polynômes de Stanley donnant les valeurs des caract\`eres sur les diagrammes de
Young multirectangulaires. 
\end{abstract}


\begin{figure}[t]
\centerline{\includegraphics{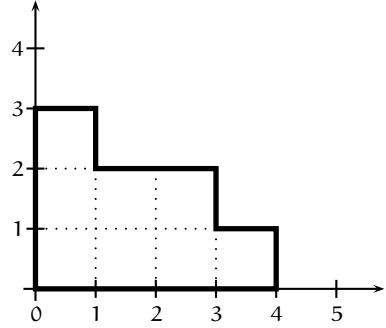}} 
\caption{Young diagram $(4,3,1)$ drawn in the French convention}
\label{fig:french}
\end{figure}

\section{Introduction}

This contribution is an extended abstract of a full
version \cite{DolegaF'eray'Sniady2008} which will be published elsewhere.

\subsection{Dilations of Young diagrams and normalized characters}

For a Young diagram $\lambda$ and an integer $s\geq 1$ we denote by
$s\lambda$ the \emph{dilation} of $\lambda$ by factor $s$. This operation can be
easily described on a graphical representation of a Young diagram: we just
dilate the picture of $\lambda$ or, alternatively, we replace each box of
$\lambda$ by a grid of $s\times s$ boxes.

Any permutation $\pi\in \Sym(k)$ can be also regarded as an element of $\Sym(n)$
if $k\leq n$ (we just declare that $\pi\in \Sym(n)$ has additional $n-k$
fixpoints).
For any $\pi\in \Sym(k)$ and an irreducible representation $\rho^\lambda$ of the
symmetric group $\Sym(n)$ corresponding to the Young diagram $\lambda$ we
define the \emph{normalized character}
\begin{displaymath}
\Sigma^\lambda_\pi =\begin{cases} \underbrace{n(n-1)\cdots
(n-k+1)}_{k \text{
factors}} \frac{\Tr \rho^\lambda(\pi)}{\text{dimension of $\rho^\lambda$}} &
\text{if } k\leq n, \\ 0 & \text{otherwise.} \end{cases} 
\end{displaymath}
We shall concentrate our attention on the characters on cycles, therefore we
will use the special notation
\begin{displaymath} 
\Sigma^\lambda_k = \Sigma^\lambda_{(1,2,\dots,k)},
\end{displaymath}
where we treat the cycle $(1,2,\dots,k)$ as an element of $\Sym(k)$ for any
integer $k\geq 1$.

The notion of dilation of a Young diagram is very useful from the viewpoint of
the asymptotic representation theory because it allows us to ask the following
question:
\begin{problem} What can
we say about the characters of the symmetric groups in the limit when the Young
diagram $\lambda$ tends in some sense to infinity in a way that $\lambda$
preserves its shape?
\end{problem} This informal problem can be formalized as follows: for
fixed $\lambda$ and $k$ we ask about (asymptotic) properties of the
normalized characters $\Sigma^{s\lambda}_{k}$ in the limit $s\to\infty$. The
reason why we decided to study this particular normalization of characters is
the following well-known yet surprising result.
\begin{fact}
For any Young diagram $\lambda$ and integer $k\geq 2$ the normalized
character on a dilated diagram
\begin{equation}
\label{eq:polynomial}
 \N\ni s \mapsto \Sigma^{s\lambda}_{k-1}
\end{equation}
is a polynomial function of degree (at most) $k$. 
\end{fact}

\subsection{Generalized Young diagrams}
\label{sec:generalized-Young-diagrams}

Any Young diagram drawn in the French convention can be identified with its
graph which is equal to the set
$\{ (x,y): 0\leq x, 0\leq y\leq f(x) \} $ for a suitably chosen function
$f:\R_+\rightarrow\R_+$, where $\R_+=[0,\infty)$, cf.~Figure \ref{fig:french}.
It is therefore natural to
define the set of \emph{generalized Young diagrams
$\Y$ (in the French convention)} as the set of bounded, non-increasing functions
$f:\R_+\rightarrow\R_+$ with a compact support; in this way any Young diagram
can be regarded as a generalized Young diagram. Notice that with the notion of
generalized Young diagrams we may consider dilations $s\lambda$ for any
real $s\geq
0$.

\subsection{How to describe the shape of a Young diagram?}
Our ultimate goal is to explicitly express the polynomials
\eqref{eq:polynomial} in terms of the shape of $\lambda$. However, before we
start this task we must ask ourselves: how to describe the shape of $\lambda$
in the best way? In the folowing we shall present two approaches to this
problem.

We define the \emph{fundamental functionals of shape} of a Young
diagram $\lambda$ by an integral over the area of $\lambda$
\begin{displaymath} 
S^\lambda_k = (k-1) \iint_{(x,y)\in\lambda} 
(\contents_{(x,y)})^{k-2}\ dx\ dy 
\end{displaymath}
for integer $k\geq 2$, and where $\contents_{(x,y)}=x-y$ is the contents of a
point of a Young diagram. When it does not lead to confusions we will skip the
explicit dependence of the fundamental functionals on Young diagrams and instead
of $S^\lambda_k$ we shall simply write $S_k$.  Clearly, each functional $S_k$ is
a homogeneous function of the Young diagram with degree $k$,
i.e.~$S^{s\lambda}_k=s^k S^{\lambda}_k$.

For a Young diagram with Frobenius coordinates 
$\lambda=\begin{bmatrix} a_1 & \cdots & a_l \\ b_1 & \cdots & b_l \end{bmatrix}
$ we define its \emph{shifted Frobenius coordinates} 
with $A_i=a_i+\frac{1}{2}$ and
$B_i=b_i+\frac{1}{2}$. Shifted Frobenius coordinates have a simple
interpretation as positions (up to the sign) of the particles and holes in the
Dirac sea corresponding to a Young diagram \cite{Okounkov2001infinite_wedge}.
Functionals $S^\lambda_k$ can be nicely expressed in terms of (shifted)
Frobenius coordinates as follows:
\begin{displaymath} 
S^\lambda_k = 
\sum_i  \int_{-\frac{1}{2}}^{\frac{1}{2}} \left[ (A_i+z)^{k-1}
- (-B_i-z)^{k-1}\right] dz, 
\end{displaymath}
\begin{equation}
\label{eq:asymptotyka-S}
 \frac{S^\lambda_k}{|\lambda|^{k-1}}= 
\sum_i  \left[ \alpha_i^{k-1} - (-\beta_i)^{k-1} \right] +
O\left(\frac{1}{|\lambda|^2} \right), 
\qquad \text{where $\alpha_i=\frac{A_i}{|\lambda|}$ and $
\beta_i=\frac{B_i}{|\lambda|}$.}
\end{equation}

Another way of describing the shape of a Young diagram $\lambda$ is to use its 
\emph{free cumulants} $R_2^\lambda,R_3^\lambda,\dots$ which are
defined as the coefficients of the
leading terms of the polynomials \eqref{eq:polynomial}:
\begin{displaymath} 
R_{k}^\lambda = [s^k] \Sigma^{s\lambda}_{k-1}=
\lim_{s\to\infty} \frac{1}{s^{k}}
\Sigma^{s\lambda}_{k-1} \qquad \text{for integer $k\geq 2$}.
\end{displaymath}
Later on we shall show how to calculate free cumulants directly from the shape
of a Young diagram. 
$R_k$ is a homogeneous function of the Young diagram with degree $k$,
i.e.~$R^{s\lambda}_k=s^k R^{\lambda}_k$.

The importance of homogeneity of $S_k^\lambda$ and $R_k^\lambda$ becomes clear
when one wants to solve asymptotic problems, such as understanding coefficients
of the polynomial \eqref{eq:polynomial}.

\subsection{Character polynomials and their applications}

It is not very difficult to show \cite{DolegaF'eray'Sniady2008} that for each
integer $k\geq 1$ there exists a polynomial with rational coefficients
$J_k(S_2,S_3,\dots)$ with a property that
\begin{displaymath} 
\Sigma_k^\lambda= J_k(S_2^\lambda,S_3^\lambda,\dots)
\end{displaymath}
holds true for any Young diagram $\lambda$. For example, we have
\begin{displaymath} 
\Sigma_{1} = S_{2}, \qquad  \Sigma_{2} = S_{3}, \qquad 
\Sigma_{3} = S_{4}-\frac{3}{2} S_{2}^2+S_{2}, \qquad 
\Sigma_{4} = S_{5}-4 S_{2} S_{3}+5 S_{3},
\end{displaymath}
\begin{displaymath} 
\Sigma_{5} = S_{6}-5 S_{2}
S_{4}-\frac{5}{2} S_{3}^2+\frac{25}{6}
S_{2}^3+15S_{4}-\frac{35}{2}S_{2}^2+8S_{2}.
\end{displaymath}

The polynomials $J_n$ are very useful, when one studies the asymptotics of
characters in the limit when the parameters
$\alpha_1,\alpha_2,\dots,\beta_1,\beta_2,\dots$ converge to some limits and the
number of boxes of $\lambda$ tends to infinity. Equation
\eqref{eq:asymptotyka-S} shows that for such scaling it is convenient to
consider a different gradation, in which the degree of $S_k$ is equal to $k-1$.
We leave it as an exercise to the Reader to use the results of this paper
to show that with respect to this gradation polynomial $J_k$ has the form
\begin{displaymath} 
\Sigma_k = S_{k+1} - \frac{k}{2} \sum_{j_1+j_2=k+1} S_{j_1}
S_{j_2} + \text{(terms of smaller degree)}. 
\end{displaymath}
The dominant part of the right-hand side (the first summand) coincides with the
estimate of Wassermann \cite{Wassermann1981} and with Thoma character on
$\Sym(\infty)$ \cite{VershikKerov1981a}. In a similar way it is possible to
obtain next terms in the expansion. 


One can also show that for each
integer $k\geq 1$ there exists a polynomial with integer coefficients
$K_k(R_2,R_3,\dots)$, called \emph{Kerov character polynomial}
\cite{Kerov2000talk,Biane2003} with a property that 
\begin{displaymath} 
\Sigma_k^\lambda= K_k(R_2^\lambda,R_3^\lambda,\dots)
\end{displaymath}
holds true for any Young diagram $\lambda$.  For example,
\begin{displaymath} 
\Sigma_1 = R_2, \qquad \Sigma_2 = R_3, \qquad
\Sigma_3 = R_4 + R_2,   \qquad
\Sigma_4 = R_5 + 5R_3,   
\end{displaymath} 
\begin{displaymath}   
\Sigma_5 = R_6 + 15R_4 + 5R_2^2 + 8R_2, \qquad
\Sigma_6 = R_7 + 35R_5 + 35R_3 R_2 + 84R_3.
\end{displaymath}
The advantage of Kerov polynomials $K_k$ over polynomials $J_k$ comes from the
fact that they usually have a much simpler form, involve smaller number of
summands and are more suitable for studying asymptotics of characters in the
case of \emph{balanced Young diagrams}, i.e.~for example in the case of
characters $\Sigma^{s\lambda}_k$ of dilated Young diagrams \cite{Biane2003}.

\subsection{The main result: explicit form of character polynomials}

For a permutation $\pi$ we denote by $C(\pi)$ the set of the cycles of $\pi$.
\begin{theorem}[Dołęga, F\'eray, Śniady \cite{DolegaF'eray'Sniady2008}]
\label{theo:rozwiniecieL}
The coefficients of polynomials $J_k$ are explicitly
described as follows:
\begin{displaymath} 
\left. \frac{\partial}{\partial S_{j_1}} \cdots
\frac{\partial}{\partial
S_{j_l}} J_k  \right|_{S_{2}=S_3=\cdots=0}  
\end{displaymath} 
is equal to $(-1)^{l-1}$ times the number of the number of the triples
$(\sigma_1,\sigma_2,\ell)$
where
\begin{itemize}
 \item $\sigma_1,\sigma_2\in\Sym(k)$ are such that $\sigma_1\circ
\sigma_2=(1,2,\dots,k)$,
 \item $\ell:C(\sigma_2)\rightarrow\{1,\dots,l\}$ is a bijective labeling,
 \item for each $1\leq i\leq l$ there are exactly $j_i-1$ cycles of $\sigma_1$
which intersect cycle $\ell^{-1}(i)$ and which do not intersect any of the
cycles $\ell^{-1}(i+1),\ell^{-1}(i+2),\dots$.
\end{itemize}

\end{theorem}

\begin{theorem}[Dołęga, F\'eray, Śniady \cite{DolegaF'eray'Sniady2008}]
\label{theo:main}
The coefficient of $R_2^{s_2} R_3^{s_3} \cdots $ in
the Kerov polynomial $K_{k}$ is equal to the number of triples
$(\sigma_1,\sigma_2,q)$ with the following properties:
\begin{enumerate}[label=(\alph*)]
 \item \label{enum:first-condition}
$\sigma_1,\sigma_2\in \Sym(k)$ are such that $\sigma_1
\circ \sigma_2=(1,2,\dots,k)$;
 \item \label{enum:ilosc2} the number of cycles of $\sigma_2$ is equal to the
number of\/ factors in
the product $R_2^{s_2} R_3^{s_3} \cdots $; in other words
$|C(\sigma_2)|=s_2+s_3+\cdots$;
 \item \label{enum:boys-and-girls} the total number of cycles of $\sigma_1$ and
$\sigma_2$ is equal to the degree of the product
$R_2^{s_2} R_3^{s_3} \cdots $; in other words
$|C(\sigma_1)|+|C(\sigma_2)|=2 s_2+3 s_3+4 s_4+\cdots$;
 \item \label{enum:kolorowanie} $q:C(\sigma_2)\rightarrow \{2,3,\dots\}$ is a
coloring of the cycles of
$\sigma_2$ with a property that each color $i\in\{2,3,\dots\}$ is used exactly
$s_i$ times (informally, we can think that $q$ is a map which to cycles of
$C(\sigma_2)$ associates the factors in the product $R_2^{s_2} R_3^{s_3}
\cdots$);
 \item \label{enum:marriage} for every set $A\subset C(\sigma_2)$
which is
nontrivial (i.e., $A\neq\emptyset$ and $A\neq C(\sigma_2)$) there are more than
$\sum_{i\in A} \big( q(i)-1 \big) $ cycles of\/ $\sigma_1$ which intersect
$\bigcup A$.
\end{enumerate}
\end{theorem}

Only condition \ref{enum:marriage} is rather complicated, therefore we will
provide two equivalent combinatorial conditions below.


\subsection{Marriage and transportation interpretations of condition
\ref{enum:marriage}}

Let $(\sigma_1,\sigma_2,q)$ be a triple which fulfills conditions
\ref{enum:first-condition}--\ref{enum:kolorowanie} of Theorem \ref{theo:main}.
We consider the following polyandrous interpretation of Hall marriage theorem.
Each cycle of $\sigma_1$ will be called a \emph{boy} and each cycle of
$\sigma_2$ will be called a \emph{girl}. For each girl $j\in C(\sigma_2)$ let
$q(j)-1$ be the \emph{desired number of husbands of $j$}. 
We say that a boy $i\in
C(\sigma_1)$ is a possible candidate for a husband for a girl $j\in C(\sigma_2)$
if cycles $i$ and $j$ intersect. Hall marriage theorem applied to our setup says
that there exists an arrangement of marriages $\M:C(\sigma_1)\rightarrow
C(\sigma_2)$ which assigns to each boy his wife (so that each girl $j$ has
exactly $q(j)-1$ husbands) if and only if 
for every set $A\subseteq C(\sigma_2)$ there are
at least $\sum_{i\in A} \big( q(i)-1 \big) $ cycles of\/ $\sigma_1$ which
intersect $\bigcup A$.
As one easily see, the above condition is similar but not identical to
\ref{enum:marriage}. The following Proposition shows the connection between
these two problems.

\begin{proposition}
Condition \ref{enum:marriage} is equivalent to each of the following two
conditions:
\begin{enumerate}[label=(\ee)]
\item \label{enum:marriage-prim} 
for every nontrivial set of girls $A\subset C(\sigma_2)$ (i.e.,
$A\neq\emptyset$ and $A\neq C(\sigma_2)$) there exist two ways of arranging
marriages $\M_p:C(\sigma_1)\rightarrow C(\sigma_2)$, $p\in\{1,2\}$ for
which the corresponding sets of husbands of wives from $A$ are different:
\begin{displaymath} 
\M_1^{-1}(A)\neq \M_2^{-1}(A), 
\end{displaymath}
\end{enumerate}
\begin{enumerate}[label=(\eeee)]
\item \label{enum:transportation}
there exists a strictly positive solution to the following system of equations: 
\begin{description}
 \item[Set of variables]\ \\ $\big\{x_{i,j}: \text{ $i\in C(\sigma_1)$
and $j\in C(\sigma_2)$ are intersecting cycles} \big\}$
 \item[Equations] $\left\{\begin{array}{l} \forall i, \sum_j x_{i,j}=1 \\
\forall j, \sum_i x_{i,j}=q(j)-1 \end{array}\right.$
\end{description}
\end{enumerate}
\end{proposition}
Note that the possibility of arranging marriages can be
rephrased as existence of a solution to the above system of equations with a
requirement that $x_{i,j}\in\{0,1\}$. 

The system of equations in condition \ref{enum:transportation} can be
interpreted as a transportation problem where each cycle of $\sigma_1$ is
interpreted as a factory which produces a unit of some ware and each cycle $j$
of $\sigma_2$ is interpreted as a consumer with a demand equal to $q(j)-1$. The
value of $x_{i,j}$ is interpreted as amount of ware transported from factory $i$
to the consumer $j$.

\subsection{General conjugacy classes}

An analogue of Theorem \ref{theo:rozwiniecieL} holds true with some minor
modifications also for the analogues of polynomials $J$ giving the values of
characters on general permutations, not just cycles. 

In case of the analogues of the Kerov polynomials giving the values of
characters on more complex permutations $\pi$ than cycles the situation is
slightly more diffcult. Namely, an analogue of Theorem \ref{theo:main} holds
true if the character $\Sigma_\pi$ is replaced by some quantities which behave
like classical cumulants of cycles constituting $\pi$ and the sum on the
right-hand side is taken only over transitive factorizations. Since the
expression of characters in terms of classical cumulants of cycles is
straightforward, we obtain an expression of characters in terms of free
cumulants.

\subsection{Applications of the main result}

The results of this article (Theorem \ref{theo:main} in particular) can be
used to obtain new asymptotic inequalities for characters of the symmetric
groups. This vast topic is outside of the scope of this article and will be
studied in a subsequent paper.

\subsection{Contents of this article}

In this article we shall prove Theorem \ref{theo:rozwiniecieL}. Also,
since the proof of Theorem \ref{theo:main} is rather long and technical
\cite{DolegaF'eray'Sniady2008}, in this overview article we shall highlight just
the main ideas and concentrate on the first non-trivial case of quadratic terms
of Kerov polynomials.

Due to lack of space we were not able to show the full history of the presented
results and to give to everybody the proper credits. For more history and
bibliographical references we refer to the full
version of this article \cite{DolegaF'eray'Sniady2008}.

\section{Ingredients of the proof of the main result}

\subsection{Polynomial functions on the set of Young diagrams}
Surprisingly, the normalized characters $\Sigma_{\pi}^{\lambda}$ can be
extended in a natural way for any generalized Young diagram $\lambda\in\Y$.
The algebra they generate will be called \emph{algebra of polynomial functions
on (generalized) Young diagrams.}  It is well-known that many natural families
of functions on Young diagrams generate the same algebra, for example the family
of free cumulants $(R_k^\lambda)$ or the family of fundamental functionals
$(S_k^\lambda)$.

\subsection{Stanley polynomials}
For two finite sequences of positive real numbers
$\p=(p_1,\dots,p_m)$
and $\q=(q_1,\dots,q_m)$ with $q_1\geq \cdots \geq q_m$ we consider a
multirectangular generalized Young diagram $\p\times \q$,
cf~Figure \ref{fig:multirectangular}. In the case when
$p_1,\dots,p_m,q_1,\dots,q_m$ are natural numbers $\p\times\q$ is a partition
\begin{displaymath}
\mathbf{p}\times \mathbf{q} =
(\underbrace{q_1,\dots,q_1}_{p_1 \text{ times}},
\underbrace{q_2,\dots,q_2}_{p_2 \text{ times}},\dots).  
\end{displaymath} 

\begin{figure}[t]
\centerline{\includegraphics[width=0.6\textwidth]{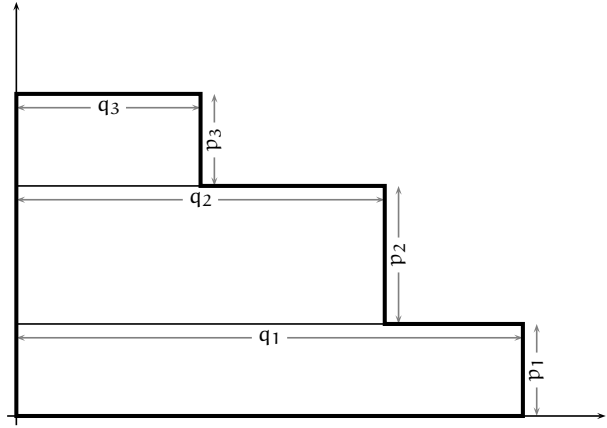}}
\caption{Generalized Young diagram $\p\times \q$ drawn in the French
convention}
\label{fig:multirectangular}
\end{figure}

\begin{proposition}
Let $\F:\Y\rightarrow\R$ be a polynomial function on the set of
generalized Young diagrams. Then $(\p,\q)\mapsto \F(\mathbf
p\times \q)$ 
is a polynomial in indeterminates
$p_1,\dots,p_m,q_1,\dots,q_m$, called \emph{Stanley polynomial}.
\end{proposition}
\begin{proof}
It is enough to prove this proposition for some family of generators of the
algebra of polynomial functions on $\Y$. In the case of functionals
$S_2,S_3,\dots$ it is a simple exercise.
\end{proof}

\begin{lemma}
\label{lem:rozwijam-esy-floresy}
If we treat $\p$ as variables and $\q$ as constants then for every $k\geq 2$ and
all $i_1<\cdots<i_s$
\begin{equation}
\label{eq:rozwiniecie-s-w-stanleya}
 [p_{i_1} \cdots p_{i_s}] S_k^{\p\times \q}= 
\begin{cases}
(-1)^{s-1}\ (k-1)_{s-1}\ q_{i_s}^{k-s} & \text{if\/ $1\leq s\leq k-1$}, \\
0 & \text{otherwise}.
\end{cases}
\end{equation}
\end{lemma}
\begin{proof} 
The integral over the Young diagram $\p\times\q$ can be split into several
integrals over rectangles constituting $\p\times\q$ therefore
\begin{multline*} 
S_k^{\p\times\q}= (k-1)
\iint_{(x,y)\in\p\times\q}(x-y)^{k-2} \ dx\ dy =\\
(k-1)! \sum_{1\leq r\leq k-1} (-1)^{r-1} \iint_{(x,y)\in\p\times\q}
\frac{x^{k-1-r}}{(k-1-r)!} \frac{y^{r-1}}{(r-1)!}\ dx \ dy = \\ 
(k-1)! \sum_{1\leq r\leq k-1} (-1)^{r-1} \sum_j
\frac{q_j^{k-r}}{(k-r)!}
\frac{(p_1+\cdots+p_j)^{r}-(p_1+\cdots+p_{j-1})^{r}}{r!}.
\end{multline*}
For any $i_1<\cdots<i_s$
\begin{displaymath}
\left.\frac{\partial^s}{\partial p_{i_1} \cdots \partial
p_{i_s}}\frac{(p_1+\cdots+p_j)^{r}-(p_1+\cdots+p_{j-1})^{r} } { r! }
\right|_{p_1=p_2=\cdots=0} =
\begin{cases} 1 & \text{if $s=r$ and $i_s=j$,} \\ 0 & \text{otherwise}
\end{cases}
\end{displaymath}
which finishes the proof.
\end{proof}

\begin{theorem}
\label{theo:stanley-and-S}
Let $\F:\Y\rightarrow\R$ be a polynomial function on the set of
generalized Young diagrams, we shall view it as a polynomial in $S_2,S_3,\dots$
Then for any $j_1,\dots,j_l\geq 2$
\begin{displaymath} 
\left. \frac{\partial}{\partial S_{j_1}} \cdots
\frac{\partial}{\partial S_{j_l}} \F \right|_{S_2=S_3=\cdots=0} = 
[p_1 q_1^{j_1-1} \cdots p_l q_l^{j_l-1}] \F^{\p \times \q}. 
\end{displaymath}
\end{theorem}
\begin{proof}
By linearity is enough to consider the case when $\F=S_{m_1} \cdots S_{m_r}$.
Clearly, the left hand side is equal to the number of permutations of the
sequence $(m_1,\dots,m_r)$ which are equal to the sequence $(j_1,\dots,j_l)$.
Lemma \ref{lem:rozwijam-esy-floresy} shows that the same holds true for the
right-hand side.
\end{proof}

\begin{corollary}
\label{coro:order-does-not-matter}
If $j_1,\dots,j_l\geq 2$ then 
\begin{displaymath} 
[p_1 q_1^{j_1-1} \cdots p_l q_l^{j_l-1}] \F ^{\p \times \q}
\end{displaymath} 
does not depend on the order of the elements of the sequence $(j_1,\dots,j_l)$.
\end{corollary}

\subsection{Stanley polynomials for characters}
The following theorem gives explicitly the Stanley polynomial for normalized
characters of symmetric groups. It was conjectured by Stanley
\cite{Stanley-preprint2006} and
proved by F\'eray \cite{F'eray-preprint2006}
and therefore we refer to it as Stanley-F\'eray character formula. 
For a more elementary proof we refer to \cite{F'eray'Sniady-preprint2007}.
\begin{theorem}
\label{theo:stanley-feray-old}
The value of the normalized character on $\pi\in \Sym(k)$ for a
multirectangular Young diagram $\p\times\q$\/ for\/ $\p=(p_1,\dots,p_r)$,
$\q=(q_1,\dots,q_r)$ is given by 
\begin{equation}
\label{eq:stanley-feray-old}
\Sigma^{\p\times \q}_{\pi}=
\sum_{\substack{\sigma_1,\sigma_2\in\Sym(k)\\ \sigma_1 \circ \sigma_2=\pi}}
\ \sum_{\phi_2:C(\sigma_2)\rightarrow\{1,\dots,r\}}
(-1)^{\sigma_1} \left[\prod_{b\in C(\sigma_1)} q_{\phi_1(b)} \prod_{c\in
C(\sigma_2)}
p_{\phi_2(c)}\right],
\end{equation}
where $\phi_1:C(\sigma_1)\rightarrow\{1,\dots,r\}$ is defined by 
\begin{displaymath} 
\phi_1(c) = \max_{\substack{b\in C(\sigma_2), \\ \text{$b$ and $c$
intersect}}} \phi_2(b).
\end{displaymath}
\end{theorem}

Notice that Theorem \ref{theo:stanley-and-S} and the above Theorem
\ref{theo:stanley-feray-old} give immediately the proof of Theorem
\ref{theo:rozwiniecieL}.

\subsection{Relation between free cumulants and fundamental functionals}

\begin{corollary}
\label{theo:stanley-feray-rrr}
The value of the $k$-th free cumulant for a
multirectangular Young diagram $\p\times\q$\/ for\/ $\p=(p_1,\dots,p_r)$,
$\q=(q_1,\dots,q_r)$ is given by 
\begin{equation}
\label{eq:stanley-feray-rrr}
R^{\p\times \q}_{k}=
\sum_{\substack{\sigma_1,\sigma_2\in\Sym(k-1)\\ \sigma_1 \circ
\sigma_2=(1,2,\dots,k-1)\\ |C(\sigma_1)|+|C(\sigma_2)|=k }}
\ \sum_{\phi_2:C(\sigma_2)\rightarrow\{1,\dots,r\}}
(-1)^{\sigma_1} \left[\prod_{b\in C(\sigma_1)} q_{\phi_1(b)} \prod_{c\in
C(\sigma_2)}
p_{\phi_2(c)}\right],
\end{equation}
where $\phi_1:C(\sigma_1)\rightarrow\{1,\dots,r\}$ is defined as in Theorem
\ref{theo:stanley-feray-old}.
\end{corollary}
\begin{proof}
It is enough to consider the homogeneous part with degree $k$ of both sides of
\eqref{eq:stanley-feray-old}  for $\pi=(1,\dots,k-1)\in\Sym(k-1)$.
\end{proof}

\begin{proposition}
\label{prop:relation-between-r-and-s}
For any integer $n\geq 2$
\begin{displaymath}R_{k} =
 \sum_{l\geq 1} \frac{1}{l!} (-k+1)^{l-1} \sum_{\substack{j_1,\dots,j_l\geq 2 \\
j_1+\cdots+j_l=k }} S_{j_1} \cdots S_{j_l}.\end{displaymath}
\end{proposition}
Before the proof notice that the above formula shows that free cumulants can be
explicitly and directly calculated from the shape of a Young diagram.

\begin{proof}
For simplicity, we shall proof a weaker form of this result, namely
\begin{equation} 
\label{eq:rands}
R_k = S_k - \frac{k-1}{2} \sum_{j_1+j_2=k} S_{j_1} S_{j_2} +
\text{(terms involving at least three factors $S_j$)}. 
\end{equation}

Theorem \ref{theo:stanley-and-S} shows that
the expansion of $R_k$ in terms of $(S_j)$ involves coefficients of Stanley
polynomials and the latter are given by Corollary \ref{theo:stanley-feray-rrr}.
We shall use this idea in the following.

Notice that the condition $|C(\sigma_1)|+|C(\sigma_2)|=k$ appearing in
\eqref{eq:stanley-feray-rrr} is equivalent to $|\sigma_1|+|\sigma_2| = |\sigma_1
\circ \sigma_2|$ where $|\pi|$ denotes the length of the permutation, i.e.~the
minimal number of factors necessary to write $\pi$ as a product of
transpositions. In other words, $\pi_1\circ \pi_2=(1,\dots,k-1)$ is a
\emph{minimal factorization} of a cycle. Such factorizations are in a bijective
correspondence with non-crossing partitions of $k-1$-element set
\cite{Biane1996}. It is therefore enough to enumerate appropriate non-crossing
partitions. We present the details of this reasoning below.

The linear term $[S_k]R_k=[p_1 q_1^{k-1}] R^{\p\times\q}_k$ is equal to the
number of minimal factorizations such that $\sigma_2$ consists of one cycle and
$\sigma_1$ consists of $k-1$ cycles. Such factorizations corresponds
to non-crossing partitions of $k-1$ element set which have exactly one block
and clearly there is only one such partition. 

Since both free cumulants $(R_j)$ and fundamental functionals of shape are
homogeneous, by comparing the degrees we see that $[S_j]R_k=0$ if $j\neq k$.
The same argument shows that $[S_{j_1} S_{j_2}] R_k=0$ if $j_1+j_2\neq k$.

Instead of finding the quadratic terms $[S_{j_1} S_{j_2}] R_k$ is better to
find the derivative $ \left. \frac{\partial^2}{\partial S_{j_1} \partial
S_{j_2}} R_k \right|_{S_{j_1}=S_{j_2}=0} $ since it better takes care of the
symmetric case $j_1=j_2$. The latter derivative is equal (up to the sign) to the
number of minimal factorizations such that $\sigma_2$ consists of two labeled
cycles $c_1,c_2$ and $\sigma_1$ consists of $k-2$ cycles. Furthermore, we
require that there are $j_2-1$ cycles of $\sigma_1$ which intersect cycle $c_2$.
This is equivalent to counting non-crossing partitions of $k-1$-element set
which consist of two labeled blocks $b_1,b_2$ and we require that the block
$b_2$ consists of $j_2-1$ elements. It is easy to see that all such non-crossing
partitions can be transformed into each other by a cyclic rotation hence there
are $k-1$ of them which finishes the proof.

The general case can be proved by analogous but more technically involved
combinatorial considerations.
\end{proof}

\subsection{Identities fulfilled by coefficients of Stanley polynomials}
The coefficients of Stanley polynomials $ [p_1^{s_1} q_1^{r_1} \cdots] \F^{\p
\times \q}$  for a polynomial function $\F$ are not linearly independent; in
fact they fulfill many identities. In the following we shall show just one of
them.

\begin{lemma}
\label{lem:tozsamosci-wielomianow-stanleya}
For any polynomial function $\F:\Y\rightarrow\R$ 
\begin{equation}
\label{eq:tozsamosc}
(j_1+j_2-1) \left[p_1 q_1^{j_1+j_2-1}\right] \F^{\p\times\q} =
- \left[p_1 p_2 q_2^{j_1+j_2-2}\right] \F^{\p\times\q}. 
\end{equation}
\end{lemma}
\begin{proof}
It is enough to prove the Lemma if $\F=S_{k_1} \dots S_{k_r}$ is a monomial
in fundamental functionals. Lemma \ref{lem:rozwijam-esy-floresy} shows that the
left-hand side
of \eqref{eq:tozsamosc} is non-zero only if $\F=S_{j_1+j_2}$ (it is also a
consequence of Theorem \ref{theo:stanley-and-S}); otherwise every monomial in
$\p$ and $\q$ with a nonzero coefficient would be at least quadratic with
respect to the variables $\p$.
The same argument shows that if the right-hand side is non-zero then either
$\F=S_{k}$ is linear (in this case $k=j_1+j_2$ by comparing the degrees) or
$\F=S_{k_1} S_{k_2}$ is quadratic. In the latter case, an inspection of the
coefficient 
\begin{displaymath} 
[p_1 p_2] S^{\p\times\q}_{k_1} S_{k_2}^{\p\times\q} = [p_1]
S^{\p\times\q}_{k_1} \cdot
[p_2] S^{\p\times\q}_{k_2} + [p_1] S^{\p\times\q}_{k_2}
\cdot [p_2] S^{\p\times\q}_{k_1} = q_1^{k_1-1} q_2^{k_2-1} + q_1^{k_2-1}
q_2^{k_1-1} 
\end{displaymath}
thanks to \eqref{eq:rozwiniecie-s-w-stanleya} leads to a contradiction.

It remains to show that for $\F=S_{j_1+j_2}$ the Lemma holds true, but this is
an immediate consequence of Lemma \ref{lem:rozwijam-esy-floresy}.
\end{proof}

\section{Toy example: Quadratic terms of Kerov polynomials}
\label{sec:toyexample}

We shall prove Theorem \ref{theo:main} in the simplest non-trivial case of the
quadratic coefficients $[R_{j_1} R_{j_2}] K_k$. In this case Theorem
\ref{theo:main} takes the following equivalent form.

\begin{theorem}
\label{theo:quadratic-reformulated}
For all integers $j_1,j_2 \geq 2$ and $k\geq 1$ the derivative 
\begin{displaymath}  
\left. \frac{\partial^2}{\partial R_{j_1} \partial R_{j_2}} K_{k}
\right|_{R_2=R_3=\cdots=0}
\end{displaymath} 
is equal to the number of triples $(\sigma_1,\sigma_2,q)$ with the following
properties:
\begin{enumerate}[label=(\alph*)]
 \item \label{enum:quadratic-a} 
$\sigma_1,\sigma_2$ is a factorization of the
cycle; in other words $\sigma_1,\sigma_2\in \Sym(k)$ are such that $\sigma_1
\circ \sigma_2=(1,2,\dots,k)$;
 \item $\sigma_2$ consists of two cycles;
 \item $\sigma_1$ consists of $j_1+j_2-2$ cycles;
 \item 
\label{enum:quadratic-c} 
$\ell:C(\sigma_2)\rightarrow \{1,2\}$ is a bijective labeling of the two cycles
of $\sigma_2$;
 \item for each cycle $c\in C(\sigma_2)$ there are at least $j_{\ell(c)}$ cycles
of $\sigma_1$ which intersect nontrivially $c$.
\end{enumerate}
\end{theorem}

\begin{proof}
Equation \eqref{eq:rands} shows that for any polynomial function $\F$ on the set
of generalized Young diagrams
\begin{displaymath} 
\frac{\partial^2}{\partial R_{j_1} \partial R_{j_2}} \F  = 
\frac{\partial^2}{\partial S_{j_1}\partial S_{j_2}} \F  + 
(j_1+j_2-1) \frac{\partial}{\partial S_{j_1+j_2}} \F,
\end{displaymath}
where all derivatives are taken at $R_2=R_3=\cdots=S_2=S_3=\cdots=0$.
Theorem \ref{theo:stanley-and-S} shows that the right-hand side is equal to
\begin{displaymath}
\left[p_1 p_2 q_1^{j_1-1} q_2^{j_2-1}\right] \F^{\p\times\q} + 
(j_1+j_2-1) \left[p_1 q_1^{j_1+j_2-1}\right] \F^{\p\times\q}. 
\end{displaymath}
Lemma \ref{lem:tozsamosci-wielomianow-stanleya} applied to the second summand
shows therefore that
\begin{equation}
\label{eq:nabula}
\frac{\partial^2}{\partial R_{j_1} \partial R_{j_2}} \F =\left[p_1 p_2
q_1^{j_1-1} q_2^{j_2-1}\right] \F^{\p\times\q} - 
\left[p_1 p_2 q_2^{j_1+j_2-2}\right] \F^{\p\times\q}. 
\end{equation}

On the other hand, let us compute the number of the triples
$(\sigma_1,\sigma_2,\ell)$ which contribute to the quantity presented in Theorem
\ref{theo:quadratic-reformulated}. By inclusion-exclusion principle it is equal
to 
\begin{multline}
\label{eq:paskuda}
\big(\text{number of triples which fulfill
conditions \ref{enum:quadratic-a}--\ref{enum:quadratic-c}}\big)+ \\
(-1) \big(\text{number of triples for which the cycle $\ell^{-1}(1)$ intersects
at most $j_1-1$ cycles of $\sigma_1$}\big) + \\
(-1) \big(\text{number of triples for which the cycle $\ell^{-1}(2)$ intersects
at most $j_2-1$ cycles of $\sigma_1$}\big). 
\end{multline}
At first sight it might seem that the above formula is not complete since we
should also add the number of triples for which the cycle $\ell^{-1}(1)$
intersects at most $j_1-1$ cycles of $\sigma_1$ and the cycle $\ell^{-1}(2)$
intersects at most $j_2-1$ cycles of $\sigma_1$, however this situation is not
possible since $\sigma_1$ consists of $j_1+j_2-2$ cycles and $\langle
\sigma_1,\sigma_2\rangle$ acts transitively.

By Stanley-F\'eray character formula \eqref{eq:stanley-feray-old} the first
summand of \eqref{eq:paskuda} is equal to 
\begin{equation}
\label{eq:rown-a}
(-1) \sum_{\substack{a+b=j_1+j_2-2,\\ 1\leq b}} \left[p_1 p_2 q_1^{a}
q_2^{b} \right] \Sigma_k^{\p\times \q},
\end{equation}
the second summand of \eqref{eq:paskuda} is equal to 
\begin{equation}
\label{eq:rown-b}
  \sum_{\substack{a+b=j_1+j_2-2,\\ 1\leq a\leq j_1-1}} 
\left[p_1 p_2 q_1^{b} q_2^{a} \right] \Sigma_k^{\p\times \q}, 
\end{equation}
and the third summand of \eqref{eq:paskuda} is equal to 
\begin{displaymath}  
\sum_{\substack{a+b=j_1+j_2-2,\\ 1\leq b\leq j_2-1}}
\left[p_1 p_2 q_1^{a} q_2^{b} \right] \Sigma_k^{\p\times \q}.
\end{displaymath}

We can apply Corollary \ref{coro:order-does-not-matter} to the
summands of \eqref{eq:rown-b}; it follows that
\eqref{eq:rown-b} is equal to 
\begin{equation}
 \label{eq:rown-d}
  \sum_{\substack{a+b=j_1+j_2-2,\\ 1\leq a\leq j_1-1}}
\left[p_1 p_2 q_1^{a} q_2^{b} \right] \Sigma_k^{\p\times \q}.
\end{equation}

It remains now to count how many times a pair $(a,b)$ contributes to
the sum of \eqref{eq:rown-a}, \eqref{eq:rown-b}, \eqref{eq:rown-d}. It is not
difficult to see that the only pairs which contribute are $(0,j_1+j_2-2)$ and
$(j_1-1,j_2-1)$, therefore the number of triples described in the formulation
of the Theorem is equal to the right-hand of \eqref{eq:nabula} which finishes
the proof.
\end{proof}


\bibliographystyle{alpha}

\bibliography{biblio2008}

\end{document}